\documentclass[10pt]{article}
\usepackage{a4wide,  amsmath, amssymb, amsfonts, amsthm, amscd,
  exscale, enumerate, dsfont}
\usepackage{psfrag}
\usepackage[latin1]{inputenc}
\usepackage[dvips]{graphicx}

\def\Br{\Bbb R}

\def\benu{\begin{enumerate}}
\def\eenu{\end{enumerate}}
\def\id{\operatorname{id}}
\def\vol{\operatorname{vol}}
\newtheorem{theorem}{Theorem}
\newtheorem{lemma}[theorem]{Lemma}

\theoremstyle{definition}
\newtheorem{dfn}{Definition}
\newtheorem{rem}[dfn]{Remark}
\newtheorem{dfnrem}[dfn]{Definition and Remark}

\begin{document}

\author{Eva Leschinsky 
\\
\\
Fakult\"at f\"ur Mathematik, Ruhr-Universit\"at Bochum,\\
D-44780 Bochum, Germany
\\
\\}

\title{Minimal geodesics and topological entropy on $T^2$}
\date{\today} 
\maketitle

\begin{abstract}
Let $(T^2, g)$ be the two-dimensional Riemannian torus. In this paper we prove that the topological entropy of the geodesic flow restricted to the set of initial conditions of minimal geodesics vanishes, independent of the 
choice of the Riemannian metric. \\
\end{abstract}

Let $(T^2, g)$ be a two-dimensional Riemannian torus. The lift of a geodesic $c$ on the universal covering is called a minimal geodesic if it is a globally minimizing geodesic. A precise definition of the topological entropy will be given later. The geodesic flow is denoted by $\phi^t$ and a geodesic with initial condition $v \in ST^2$ is denoted by $c_v$. By construction it holds $\dot c_v(t)=\phi^t(v)$.
Our aim is to prove the following Main Theorem:\\

{\bf Main Theorem.} \label{min}
{\it Let $(T^2, g)$ be a two-dimensional torus with a Riemannian metric and $\Br^2$ its universal covering.
Let $\tilde S\Br^2$ be defined as follows:
$$\tilde S\Br^2:=\{v \in S\Br^2\, |\, c_v \text{ is a minimal geodesic } \}$$
and $\tilde ST^2:= d_p(\tilde S\Br^2)$, where $p:\Br^2 \to T^2$ is the covering map.
Let $\tilde \phi^t$ be the restriction of the geodesic flow $\phi^t$ on the unit tangent bundle of $T^2$ to $\tilde ST^2$. Then,
$$h_{top}(\tilde \phi^t)=0.$$}

\noindent
First we have to introduce some definitions and results for minimal geodesics on $T^2$.

\begin{dfn}
Let $(Y,d)$ be a compact metric space and let $\phi^t: Y \to Y$ be a continuous flow. Then, for given $T\geq 0$ the {\it dynamical distance function} is
defined as 
$$d(v,w)_{T}:= \max_{0 \leq t \leq T}d(\phi^tv, \phi^tw)$$
for all $v,w \in Y$.\\
Two distance functions $d_1$, $d_2$ on $Y$ are called {\it uniformly equivalent}, if $$\id:(Y, d_1) \to (Y, d_2) \quad \text{and} \quad \id:(Y,d_2) \to (Y,d_1)$$ are both uniformly continuous.
\end{dfn}

\begin{rem}[Distance functions on $\tilde ST^2$] \label{dist}
The Riemannian metric $g$ on $T^2$ induces the Sasaki metric on $ST^2$. By this we get a distance function $\tilde d$ on $\tilde ST^2$. 
Let $d$ denote the distance function on the universal covering $\Br^2$ of $T^2$ induced by $g$. Then $\bar d$ is another metric on $\tilde ST^2$ defined by 
$$\bar d(v,w):=\max_{t \in [0,1]} d (c_v(t), c_w(t)).$$
The distance functions $\tilde d$ and $\bar d$ are uniformly equivalent. \\
Then, for given $T \geq 0$ the dynamical distance functions on $\tilde ST^2$ with $v, w \in \tilde ST^2$ are
$$\tilde d(v,w)_{T}= \max_{0 \leq t \leq T} \tilde d(\phi^tv, \phi^tw)$$
and
$$\bar d(v,w)_{T} = \max_{0 \leq t \leq T+1} d (c_v(t), c_w(t)).$$
\end{rem}

\begin{dfn}[Topological entropy] \label{topent}
Let $(Y,d)$ be a compact metric space and let $\phi^t: Y \to Y$ be a continuous flow. We fix $\varepsilon>0$. \\
A subset $F\subset Y$ is called a $(\phi,\varepsilon)$-{\it separated} set of $Y$ with respect to $T$, if for $x_1 \not=x_2 \in F$ it holds $d(x_1,x_2)_{T}> \varepsilon.$ \\
A subset $G\subset Y$ is called a $(\phi,\varepsilon)$-{\it spanning} set of $Y$ with respect to $T$, if for each $x \in Y$ there exists a $y \in F$ with $d(x,y)_{T} \leq \varepsilon.$
The topological entropy of $\phi^t$ is defined by Bowen as 
$$
h_{top}(g) = h_{top}(\phi)=\lim_{\varepsilon \to 0} \limsup_{T \to \infty} \left(\frac{1}{T} \log r_T(\phi, \varepsilon)  \right)=  \lim_{\varepsilon \to 0} \limsup_{T \to \infty} \left(\frac{1}{T} \log s_T(\phi, \varepsilon)  \right).
$$
Here $r_T(\phi, \varepsilon)$ denotes the maximal cardinality of any $(\phi,\varepsilon)$-separated set and $s_T(\phi, \varepsilon)$ denotes the minimal cardinality of any $(\phi,\varepsilon)$-spanning set of $Y$ with respect to $T$.\\
When we restrict ourselves to a compact subset $K\subset Y$ and elect separated and spanning sets with respect to $K$, then we denote the topological entropy restricted to $K$ by $h_{top}(\phi, K)$. Obviously it holds $h_{top}(\phi, K) \leq h_{top}(\phi)$. For a later theorem we also need the expression
$$h_{top}(\phi, \varepsilon)=\limsup_{T \to \infty} \left(\frac{1}{T} \log r_T(\phi, \varepsilon)  \right)=  \limsup_{T \to \infty} \left(\frac{1}{T} \log s_T(\phi, \varepsilon)  \right).
$$
For more details see for example \cite{KH} or \cite{Wo}. 
\end{dfn}

\begin{rem}
As the topological entropy of a continuous flow is independent of the choice of uniformly equivalent distance functions, we consider the distance function $\bar d$ instead of $\tilde d$ on $\tilde ST^2$.
\end{rem}

\begin{dfnrem} \label{background}
Let $g$ be the induced Riemannian metric on the universal covering $\Br^2$ of $T^2$. We fix a Euclidian metric $g_E$ on $\Br^2$ which is automatically equivalent to $g$ and call it the {\it background metric}. There exists a constant $A>0$ such that
for the distance functions $d(\cdot, \cdot)$ and $d_E(\cdot, \cdot)$ induced by $g$ and $g_E$ on the universal covering $\Br^2$ it holds
$$\frac{1}{A} d_E(x,y) \leq d(x,y) \leq A d_E(x,y)$$
for all $x,y \in \Br^2$.
\end{dfnrem}

\begin{rem}[Minimal geodesics and Euclidian lines] \label{Hedlund}
Hedlund and Morse (see~\cite{Hedlund}) proved that there exists a constant $D>0$, such that for each minimal geodesic $c$ on the universal covering $\Br^2$ there exists a Euclidian line $l_c$, and for each Euclidian line $l_c$ there exists a minimal geodesic $c$ such that 
$$d (l_c, c(t)) \leq D, \quad \text{for all $t \in \Br.$}$$
The Euclidian slope of $l_c$ implies for each minimal geodesic $c$ a rotation number $ \alpha \in \Br \cup \{\infty\}$. Let $\mathcal{M}_{\alpha}$ denote the set of all minimal geodesics on $\Br^2$ with a fixed rotation number $\alpha$. \\
As shown by Bangert (see~\cite{B1988}), for irrational $\alpha$ the set $\mathcal{M}_{\alpha}$ is totally ordered, i.e. all minimal geodesics in $\mathcal{M}_{\alpha}$ have pairwise no intersections with each other. \\
For $\alpha$ rational ($\alpha = \infty$ is rational) let $\mathcal{M}_{\alpha} ^{\operatorname{per}}$ denote the set of periodic minimal geodesics with rotation number $\alpha$. Obviously, by the minimality of the geodesics two arbitrary elements in $\mathcal{M}_{\alpha} ^{\operatorname{per}}$ do not intersect. 
We distinguish two cases: \\
1) The periodic geodesics in a closed geodesic strip are foliated, i.e. for each point $x$ in the strip there exists a minimal geodesic $c$ with rotation number $\alpha$ and $c(0)=x$.\\
2) The periodic geodesics in a closed geodesics strip are not foliated, i.e. there exist gaps bounded by so called neighboring minimal geodesics.  \\
With respect to some introduced order on $\mathcal{M}_{\alpha} ^{\operatorname{per}}$ in $\Br^2$, let $x^-<x^+$ be two neighboring elements of $\mathcal{M}_{\alpha} ^{\operatorname{per}}$. Neighboring means then, that there exists no $x \in \mathcal{M}_{\alpha} ^{\operatorname{per}}$ with $x^-<x<x^+$. Let 
$$\mathcal{M}_{\alpha}^+(x^-, x^+):= \{x \in \mathcal{M}_{\alpha}\;|\;\text{$x$ is $\alpha$-asymptotic to $x^-$ and $\omega$-asymptotic to $x^+$}\},$$
$$\mathcal{M}_{\alpha}^-(x^-, x^+):= \{x \in \mathcal{M}_{\alpha}\;|\;\text{$x$ is $\omega$-asymptotic to $x^-$ and $\alpha$-asymptotic to $x^+$}\}.$$
The sets $\mathcal{M}_{\alpha}^+$, $\mathcal{M}_{\alpha}^-$ denote the unions of $\mathcal{M}_{\alpha}^+(x^-, x^+)$ and $\mathcal{M}_{\alpha}^-(x^-, x^+)$ extended over all pairs of neighboring elements.\\
Then, for rational $\alpha$ the set $\mathcal{M}_{\alpha}$ is the disjoint union of $\mathcal{M}_{\alpha} ^{\operatorname{per}}$, $\mathcal{M}_{\alpha}^+$ and $\mathcal{M}_{\alpha}^-$, and the sets $\mathcal{M}_{\alpha} ^{\operatorname{per}} \cup \mathcal{M}_{\alpha}^+$ and $\mathcal{M}_{\alpha} ^{\operatorname{per}} \cup \mathcal{M}_{\alpha}^-$ are ordered. For more details see~\cite{B1988}. 
\end{rem}

To prove the Main Theorem we have to introduce a construction on the universal covering. This construction goes back to Manning in \cite{Manning}. The central idea of the proof of the Main Theorem is then Theorem~\ref{Theo-Bowen} of R.~Bowen which allows us to bound $h_{top}(\tilde \phi)$ by the sum of the topological entropy of strips with finite width and the topological entropy of a single strip. \\

\noindent 
{\bf Central construction}\\
Let $\mathcal{F}$ be a fundamental domain in $\Br^2$ of diameter $a$, $\overline{\mathcal{F}}$ its closure and let $\varepsilon >0$ be a small number. We choose $r>0$ large with respect to the constants $a$ and $D$, and consider the set
$$\mathcal{F}_r:=\{z \in \Br^2 \, |\, r-a \leq d(z,\overline{\mathcal{F}}) \leq r\}. $$
For each $x \in \mathcal{F}$ the set $\mathcal{F}_r$ is obviously contained in the closed ball $\bar B(x,r+a)$ (see Figure~\ref{fig1}).\\
\begin{figure}[h]
\begin{center}
\psfrag{Fre}{$ \in \mathcal{F}_r^{\varepsilon}$}
\psfrag{Fe}{$\mathcal{F}^{\varepsilon}$}
\psfrag{Fr}{$\mathcal{F}_r$}
\psfrag{r}{$r+a$}
\psfrag{F}{$\mathcal{F}$}
\psfrag{B}{$B(x,r+a)$}
\psfrag{e}{$\frac{\varepsilon}{2}$}
\includegraphics[scale=0.7]{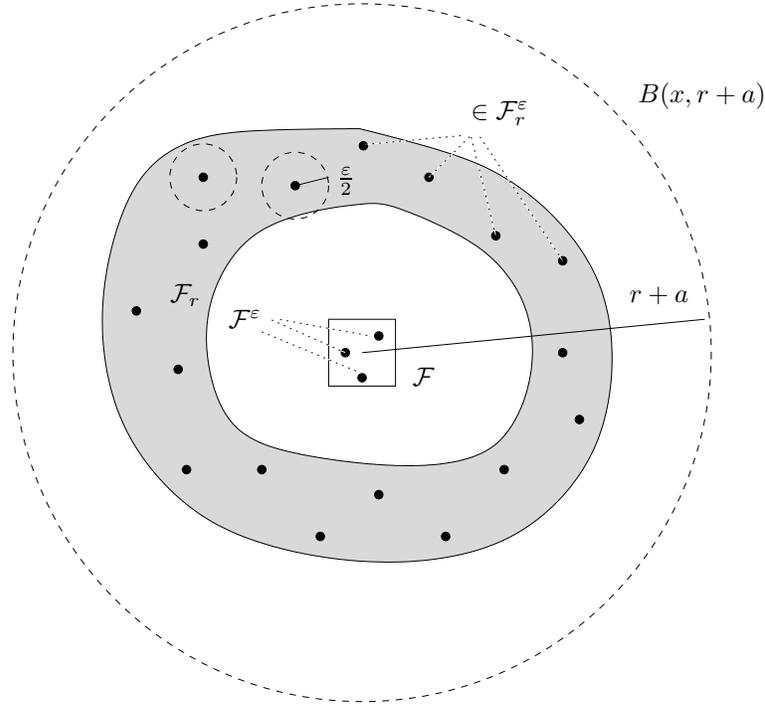}
\end{center}
\caption{\label{fig1} Illustration of the different separated sets.} 
\end{figure}
\noindent
Let $\mathcal{F}_r^{\varepsilon}$ be some $(d, \varepsilon)$-separated set of $\mathcal{F}_r$ with maximal cardinality.\\ 
Let $C_{\varepsilon}:=\inf\limits_{y \in \Br^2} \vol \left(B(y, \frac{\varepsilon}{2})\right) = \inf\limits_{y \in \mathcal{F}} \vol \left(B(y, \frac{\varepsilon}{2})\right)$. Then,
\begin{equation} \label{Ball}
\# \mathcal{F}_r^{\varepsilon} \cdot C_{\varepsilon} \leq \vol \left(B\left(x,r+a+\frac{\varepsilon}{2}\right)\right).
\end{equation}
Let $\mathcal{F}^{\varepsilon}$ be some $(d, \varepsilon)$-separated set of $\mathcal{F}$ with maximal cardinality. \\
We consider the flat background metric as introduced in Definition~\ref{background}. Obviously, geodesics on the universal covering with respect to $g_E$ are straight lines.
For $y \in\mathcal{F}^{\varepsilon}$ and $z \in \mathcal{F}_r^{\varepsilon}$ let $l_{yz}$ be a straight line, joining $y$ and $z$. As mentioned in Remark~\ref{Hedlund} there exists a minimal geodesic $c_{yz}$ for $l_{yz}$, and a constant $D>0$ such that $d(l_{yz},c_{yz}(t)) \leq D$ for all $t \in \Br$. We choose $v_{yz}$ as an initial vector of $c_{yz}$, such that $\pi(v_{yz}) \in \mathcal{F}$ and re-parametrize $c_{yz}$ such that $\dot c_{yz}(0)=v_{yz}$. Obviously $v_{yz} \in \tilde ST^2$.\\
Let $$P_r:=\{v_{yz} \in \tilde ST^2\, |\, y \in \mathcal{F}^{\varepsilon}, z \in \mathcal{F}_r^{\varepsilon}\}$$ be the set of initial conditions of minimal geodesics (parametrized by arc length) which have the same rotation number as the connecting straight lines between $y \in \mathcal{F}^{\varepsilon}$ and $z \in \mathcal{F}_r^{\varepsilon}$, and which are constructed as described above.

\begin{theorem} \label{vieleStreifen-vorab}
There exists a constant $\beta > 0$ independent of $r$, such that $P_r$ is a $(r, \beta)$-spanning set of $\tilde ST^2$ with respect to $\tilde \phi^t$ and the metric $\bar d$, i.e. for each $w \in \tilde ST^2$ there exists $v \in P_r$ with $\bar d(v,w)_{\tilde \phi,r} \leq \beta$. 
\end{theorem}

\begin{lemma} \label{vieleStreifen}
$h_{top}(\tilde \phi_t, \beta)=0$.
\end{lemma}

\begin{lemma}\label{p}
Let $c: \Br \to \Br^2$ be a geodesic and $l$ a Euclidian straight line on $\Br^2$ endowed with an order. Then any map $p:c(\Br) \to l$ with
$$d(c(t), p(c(t))=d(c(t), l)$$
is an injective and strictly monotone function.
\end{lemma}

\begin{proof}
Assume $p$ is not injective, then there exist $t_1 \not= t_2 \in \Br$ with $c(t_1) \not= c(t_2)$ and $x:=p(c(t_1))=p(c(t_2))$. Let $\gamma_1:[0,a] \to \Br^2$ and $\gamma_2:[0,b] \to \Br^2$ denote the minimal geodesic segments connecting $c(t_1)$ with $p(c(t_1))$ and $c(t_2)$ with $p(c(t_2))$, respectively. $\dot \gamma_1(a)$ and $\dot \gamma_2(b)$ are both orthogonal to $l$ in $x$, otherwise the distance of $c(t_1)$ or $c(t_2)$ to $l$ could be shortened. But then $\dot \gamma_1(a)=\dot \gamma_2(b)$ and this implies that $\gamma_1(\Br)=\gamma_2(\Br)$. If $a=b$ then $\gamma_1(0)=\gamma_2(0)$, in contradiction to $c(t_1) \not=c(t_2)$. Let $a\not=b$ and w.l.o.g.\ let $\gamma_1(0)=\gamma_2(b')$ for $0 < b' <b$. But then $\gamma_2([0,b])$ and $c(\Br)$ intersect twice, in contradiction to their minimality. Hence, $p$ is injective.\\

\noindent
Assume $p$ is not monotone, then there exist $t_1<t_2<t_3$ such that $p(c(t_1))<p(c(t_3))<p(c(t_2))$. But then at least two of the connecting segments $\gamma_1, \gamma_2, \gamma_3$ intersect. W.l.o.g.\ let $\gamma_1:[0,a] \to \Br^2$ and $\gamma_2:[0,b] \to \Br^2$ intersect. Then there exist times $\tilde a$ and $\tilde  b$ with $\gamma_1(\tilde a)=\gamma_2(\tilde b)$. It holds $a- \tilde a= b - \tilde b$, otherwise we get a contradiction to the minimality of the segments $\gamma_1$ and $\gamma_2$. But then 
$$\gamma:[0,a] \to \Br^2\quad \text{with}\quad \gamma(t)=
\begin{cases} 
	\gamma_1(t) & t \leq t\tilde a\\ 
	\gamma_2(t-\tilde a + \tilde b) & t> \tilde a 
\end{cases}$$
has the same length as $\gamma_1$. Because of the kink in $\gamma(\tilde a)$ this piecewise geodesics segment can be shortened, in contradiction to $d(c(t_1), p(c(t_1))=d(c(t_1), l)$. Hence, $p$ is monotone and by the injectivity strictly monotone.
\end{proof}

\begin{proof}[\bf{Proof of Theorem~\ref{vieleStreifen-vorab}}]
We choose an arbitrary $w \in \tilde S T^2$. Let $c_1$ be the minimal geodesic with $\dot c_1(0) = w$. Hence, it holds that $c_1(0) \in \mathcal{F}$ and $c_1(r) \in\mathcal{F}_r$.
As $\mathcal{F}^{\varepsilon}$ and $\mathcal{F}_r^{\varepsilon}$ are $\varepsilon$-separating sets with maximal cardinality, there exist $y \in \mathcal{F}^{\varepsilon}$ and $z \in \mathcal{F}_r^{\varepsilon}$ with
\begin{eqnarray} \label{c1z}
d (c_1(0),y) \leq \varepsilon\quad \text{and} \quad d (c_1(r),z) \leq \varepsilon.
\end{eqnarray}
We consider the straight line $l_2$, joining $y$ and $z$, and the corresponding geodesic $c_2$ with $\dot c_2(0) = v_{yz}$ and $v_{yz} \in P_r$. By $l_1$ we denote 
a straight line which accompanies $c_1$ (see Remark~\ref{Hedlund}), satisfying $d(c_1(t), l_1) \leq D$ for all $t \in \Br$. Let $p_i: c_i(\Br) \to l_i$, $i \in \{1,2\}$ be two maps such that 
$$d(c_2(t),p_2(c_2(t)))=d(c_2(t), l_2)\leq D\quad \text{and} \quad d(c_1(t),p_1(c_1(t)))=d(c_1(t),l_1)\leq D.$$ 
By Lemma~\ref{p} these maps are injective and strictly monotone.
For an illustration of the introduced objects see Figure~\ref{fig2}.\\

\begin{figure}[h]
\begin{center}
\psfrag{Fr}{$\mathcal{F}_r$}
\psfrag{y}{$y$}
\psfrag{z}{$z$}
\psfrag{cw}{$c_1$}
\psfrag{l}{$l_1$}
\psfrag{cvt}{$c_2(t)$}
\psfrag{cwt}{$c_1(t)$}
\psfrag{c2r}{$c_2(r)$}
\psfrag{c2tr}{$c_2(\tilde r)$}
\psfrag{ptl}{$p_2(c_2(t))$}
\psfrag{qtl}{$p_1(c_1(t))$}
\psfrag{lyz}{$l_2$}
\psfrag{cvyz}{$c_2$}
\psfrag{c0}{$c_1(0)$}
\psfrag{cr}{$c_1(r)$}
\psfrag{F}{$\mathcal{F}$}
\psfrag{B}{$B(x,r+a)$}
\psfrag{le}{$\leq \varepsilon$}
\includegraphics[scale=0.9]{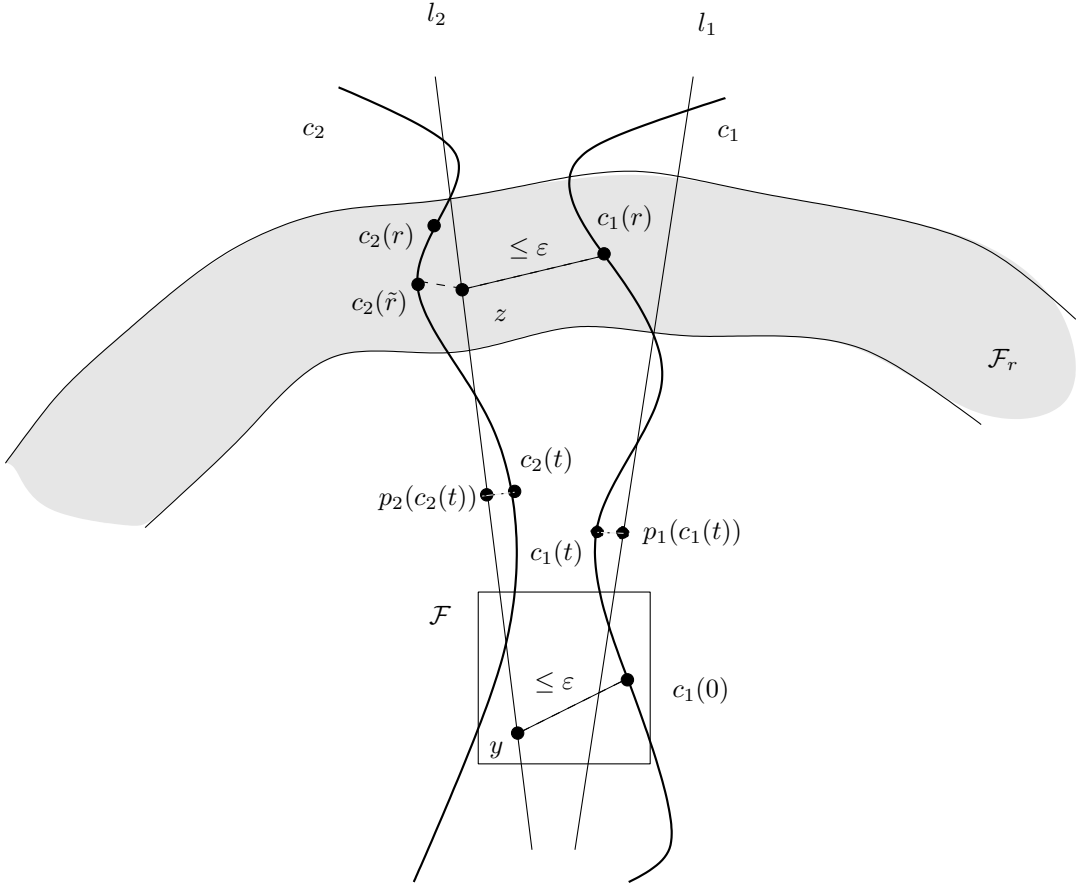}
\end{center}
\caption{\label{fig2} Illustration of $l_1$, $l_2$ and $c_1$, $c_2$ in the proof of Theorem~\ref{vieleStreifen-vorab}.} 
\end{figure}

\noindent
We have to prove that the distance of $c_1(t)$ and $c_2(t)$ is bounded for all $t \in [0,r+1]$ and that the bound is independent of $r$ and $t$.
As there exist the accompanying lines $l_1$ and $l_2$ the main part of the proof will be to show that the distance of the points $p_1(c_1(t))$ and $p_2(c_2(t))$ is bounded. For this we will consider finite segments $s_1, s_2$ of $l_1$ and $l_2$ which contain the images of the maps $p_1$ and $p_2$, respectively. Then we will use that the distances of the end points of $s_1$ and $s_2$ with respect to the Euclidian background metric give us bounds for the distances of the segments in between.\\

In the first step we want to bound $d(z, c_2(r))$. Let $\tilde r$ be a time with $d(c_2(\tilde r), z) \leq D$. Then by the triangle-inequality: 
$$\tilde r= d(c_2(\tilde r), c_2(0)) \leq d(c_2(\tilde r), z) + d(z,c_1(r)) + d(c_1(r), c_1(0)) + d(c_1(0),c_2(0)) \stackrel{(\ref{c1z})}{\leq} D+\varepsilon + r+a$$
and
$$r = d(c_1(r), c_1(0)) \leq d(c_1(r), z) + d(z,c_2(\tilde r)) + d(c_2(\tilde r), c_2(0)) + d(c_2(0), c_1(0)) \stackrel{(\ref{c1z})}{\leq}
\varepsilon + D + \tilde r + a.$$
Both inequalities together imply
$$ |r - \tilde r| \leq D + \varepsilon + a.$$
By this we get 
\begin{equation} \label{zc2}
d(z, c_2(r)) \leq d(z,c_2(\tilde r)) + d(c_2(\tilde r), c_2(r)) \leq D + |\tilde r-r| \leq D+D+\varepsilon +a= 2D+\varepsilon +a.
\end{equation}
Now we can bound the distance of $p_2(c_2(r))$ and $p_1(c_1(r))$:
\begin{eqnarray}
d(p_2(c_2(r)), p_1(c_1(r))) &\leq & d(p_2(c_2(r)), c_2(r)) + d(c_2(r), c_1(r)) + d(c_1(r), p_1(c_1(r))) \nonumber\\
&\leq & D + d(c_2(r),z) + d(z, c_1(r)) + D\nonumber\\
& \stackrel{(\ref{c1z})+(\ref{zc2})}{\leq} & D + (2D+\varepsilon +a) + \varepsilon + D \nonumber\\
&=& 4D + a + 2\varepsilon \label{r+1+2D}.
\end{eqnarray}
A much easier task is to find an upper bound for the distance between $p_2(c_2(0))$ and $p_1(c_1(0))$:
\begin{eqnarray}
d(p_2(c_2(0)), p_1(c_1(0)) &\leq & d(p_2(c_2(0)), c_2(0)) + d(c_{2}(0), c_1(0)) + d(c_1(0), p_1(c_1(0)))\nonumber\\
&\leq & D + a + D \nonumber\\
&= & 2D+a \label{-2D}.
\end{eqnarray}
We want to fix finite segments $s_1,s_2$ of $l_1$ and $l_2$ such that 
$$p_1(c_1([0,r+1])) \subset s_1\quad \text{and} \quad p_2(c_2([0,r+1])) \subset s_2.$$
By Lemma~\ref{p} we choose as $s_1$ the connected segment between $p_1(c_1(0))$ and $p_1(c_1(r+1))$, and as $s_2$ the connected segment between $p_2(c_2(0))$ and $p_2(c_2(r+1))$, respectively.
The next Lemma will allow us to bound the distance of $c_1(t)$ and $c_2(t)$ for all $t \in [0,r+1]$ independently of $r$ and $t$:

\begin{lemma}\label{GeradenLemma}
There exists a constant $H:=2(A^2(4D+a+2\varepsilon) + 2D)>0$ such that 
\begin{equation} \label{H}
d(p_2(c_2(t)), p_1(c_1(t))) \leq H, \quad \text{for all $t \in [0,r+1]$}
\end{equation}
\end{lemma}

Before we give a proof of this Lemma, we finish the proof of the Theorem~\ref{vieleStreifen-vorab}:\\
By Lemma~\ref{GeradenLemma} we get for all $t \in [0, r+1]$
\begin{eqnarray*}
d(c_2(t), c_1(t)) &\leq & d(c_2(t), p_2(c_2(t))) +d(p_2(c_2(t)), p_1(c_1(t))) + d(p_1(c_1(t)), c_1(t))\\
& \stackrel{(\ref{H})}{\leq} & D + H + D \\
&=& 2D +H.
\end{eqnarray*}
This implies the statement of the Theorem:
$$\bar d(w, v_{yz})_{\tilde \phi, r}= \max_{0 \leq t \leq r+1} d (c_w(t), c_{v_{yz}}(t))= \max_{0 \leq t \leq r+1} d(c_1(t), c_2(t)) \leq 2D+H.$$
We set $\beta:=\beta(D, A, a, \varepsilon) = 2D+H = 6D + 2(A^2(4D+a+2 \varepsilon) + 2D) = 10D + 2A^2(4D+a+2 \varepsilon)$.\\

\noindent
Hence, $P_r$ is a $(r,\beta)$-spanning set of $\tilde S T^2$ with respect to $\tilde \phi^t$ and the distance function $\bar d$.
\end{proof}

\begin{proof}[\bf{Proof of Lemma~\ref{GeradenLemma}}]
To prove this statement we will first show that for each $t \in [0,r+1]$ there exists $\tilde t \in \Br$ such that 
\begin{eqnarray} \label{B}
d(p_2(c_2(t)), p_1(c_1(\tilde t))) \leq B := A^2 (4D + a + 2\varepsilon) +2D
\end{eqnarray}
and 
\begin{eqnarray} \label{t-t}
|t- \tilde t| \leq 2D+B+a.
\end{eqnarray}
Then we will conclude the statement of the Lemma.\\

\noindent
As introduced before let $d_E(\cdot, \cdot)$ denote the Euclidian distance function and $A>0$ the equivalence constant between the Euclidian and Riemannian distance. We consider two affine linear functions $f_1,f_2:[0,1] \to \Br^2$ with 
$$f_1(k) = (p_1(c_1(r+1))-p_1(c_1(0)))k + p_1(c_1(0))$$ and 
$$f_2(k) = (p_2(c_2(r+1))-p_2(c_2(0)))k + p_2(c_2(0)).$$ 
Obviously $f_1([0,1])=s_1$ and $f_2([0,1]) = s_2$.
Then $d_E(f_1(k), f_2(k)) = ||f_1(k)-f_2(k)||$ is a convex and continuous function. Hence, it has its maximum on the boundary, i.e. 
\begin{eqnarray}
d_E(f_1(k), f_2(k)) &\leq& \max\{ d_E(f_1(0), f_2(0)), d_E(f_1(1), f_2(1))\} \nonumber\\
&=& \max\{d_E(p_1(c_1(0)),p_2(c_2(0))), d_E(p_1(c_1(r+1)),p_2(c_2(r+1)))\} \nonumber\\
&\leq& A \cdot \max\{d(p_1(c_1(0)),p_2(c_2(0))), d(p_1(c_1(r+1)),p_2(c_2(r+1)))\}\nonumber\\
& \stackrel{(\ref{-2D})+(\ref{r+1+2D})}{\leq} & A(4D + a + 2\varepsilon) \label{f1kf2k}
\end{eqnarray}
For fixed $t \in [0, r+1]$ we choose $k \in [0,1]$ such that $p_2(c_2(t))=f_2(k)$. There exists $\tilde t \in \Br$ 
such that $d(f_1(k), c_1(\tilde t)) \leq D$. By construction it holds $d(p_1(c_1(\tilde t)), c_1(\tilde t)) \leq D$ and obviously, 
$$d(f_1(k), p_1(c_1(\tilde t)))\leq d(f_1(k), c_1(\tilde t)) + d(c_1(\tilde t), p_1(c_1(\tilde t))) \leq 2D.$$
This implies that 
\begin{eqnarray}
d(p_2(c_2(t)), p_1(c_1(\tilde t))) &\leq& d(p_2(c_2(t)), f_1(k)) + d(f_1(k), p_1(c_1(\tilde t))) \nonumber\\
&\leq&  A d_E(p_2(c_2(t)), f_1(k))  + 2D \nonumber\\
&=& A d_E(f_2(k), f_1(k)) +2D \nonumber\\
&\stackrel{(\ref{f1kf2k})}{\leq}& A^2 (4D + a + 2\varepsilon) +2D. \label{t+t}
\end{eqnarray}
\begin{figure}[h]
\begin{center}
\psfrag{Fr}{$\mathcal{F}_r$}
\psfrag{cw}{$c_1$}
\psfrag{l}{$l_1$}
\psfrag{cvt}{$c_2(t)$}
\psfrag{cwt}{$c_1(t)$}
\psfrag{ptl}{$p_2(c_2(t))$}
\psfrag{qtl}{$p_1((c_1(\tilde t))$}
\psfrag{lyz}{$l_2$}
\psfrag{cvyz}{$c_2$}
\psfrag{c0}{$c_1(0)$}
\psfrag{cv0}{$c_2(0)$}
\psfrag{p01}{$p_1(c_1(0))$}
\psfrag{p02}{$p_2(c_2(0))$}
\psfrag{F}{$\mathcal{F}$}
\psfrag{cvrd}{$c_2(r+1)$}
\psfrag{cwrd}{$c_1(r+1)$}
\psfrag{pr2}{$p_2(c_2(r+1))$}
\psfrag{pr1}{$p_1(c_1(r+1))$}
\psfrag{cvd}{$c_2(-2D)$}
\psfrag{cwd}{$c_1(-2D)$}
\psfrag{f2k}{$f_1(k)$}
\psfrag{cwtt}{$c_1(\tilde t)$}
\includegraphics[scale=0.76]{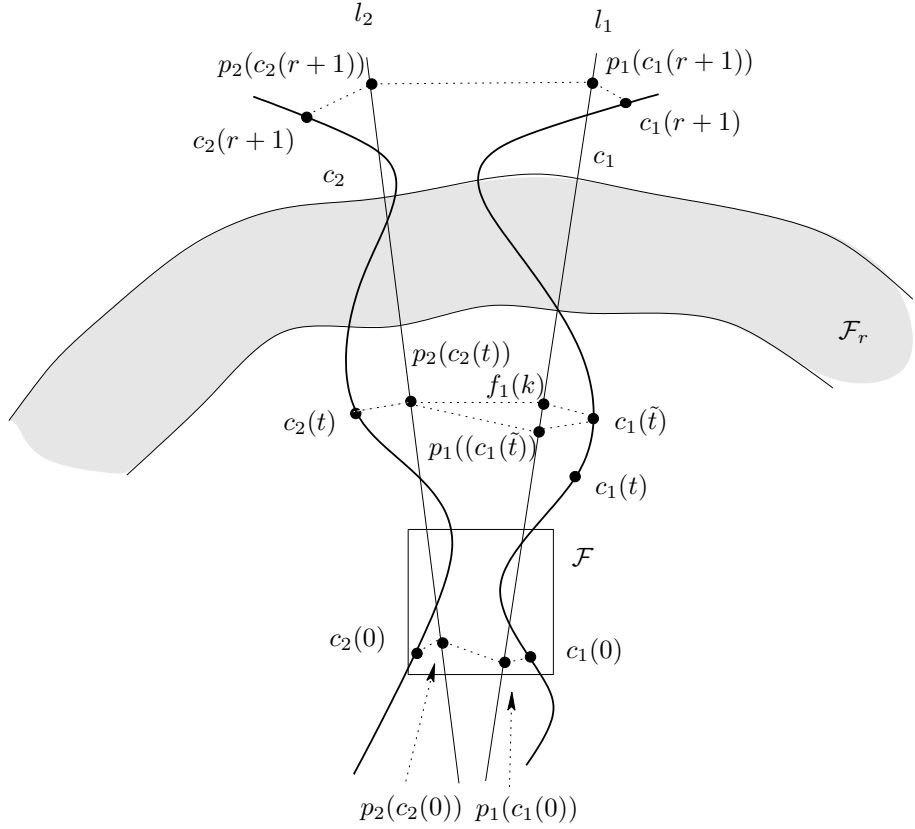}
\end{center}
\caption{\label{fig4} Illustration of notions introduced in Lemma~\ref{GeradenLemma}.} 
\end{figure}
Now we will bound $|t - \tilde t|$:
\begin{eqnarray*}
\tilde t = d(c_1(\tilde t), c_1(0)) &\leq & d(c_1(\tilde t), p_1(c_1(\tilde t))) + d(p_1(c_1(\tilde t)), p_2(c_2(t))) \\
& & + d( p_2(c_2(t)), c_2(t)) + d(c_2(t), c_2(0)) + d(c_2(0), c_1(0))\\
&\stackrel{(\ref{t+t})}{\leq} & D + B + D + t + a
\end{eqnarray*}
and
\begin{eqnarray*}
t= d(c_2(t), c_2(0)) & \leq & d(c_2(0), c_1(0)) + d( c_1(0), c_1(\tilde t))\\
& & + d( c_1(\tilde t), p_1(c_1(\tilde t))) + d(p_1(c_1(\tilde t)), p_2(c_2(t))) + d(p_2(c_2(t)), c_2(t))\\
&\stackrel{(\ref{t+t})}{\leq}& a+\tilde t + D + B + D.
\end{eqnarray*}
Both inequalities together imply 
\begin{eqnarray*}
|t-\tilde t| \leq 2D+B+a. 
\end{eqnarray*}

We can now conclude that
\begin{eqnarray*}
d(p_2(c_2(t)), p_1(c_1(t))) &\leq & d(p_2(c_2(t)), p_1(c_1(\tilde t))) + d(p_1(c_1(\tilde t)), c_1(\tilde t))\\
& & + d(c_1(\tilde t), c_1(t)) + d(c_1(t), p_1(c_1(t)))\\
& \stackrel{(\ref{B})+(\ref{t-t})}{\leq} & B + D + (2D+B+a) + D\\
&=& 4D + 2(A^2(4D+a+2 \varepsilon) + 2D)\\
&=:& H.
\end{eqnarray*}
\end{proof}

\begin{proof}[\bf{Proof of Lemma~\ref{vieleStreifen}}]
Let $\tilde P_r$ be a $(r, \beta)$-spanning set of $\tilde ST^2$ of minimal cardinality with respect to $\tilde \phi^t$ and $\bar d$. Then:
\begin{eqnarray*}
h_{top}(\tilde \phi, \beta) &=& \limsup_{r \to \infty} \frac{1}{r} \log (\#\tilde P_r) \\
&\leq& \limsup_{r \to \infty}\frac{1}{r} \log(\#P_r) \qquad \qquad \text{(by Theorem~\ref{vieleStreifen-vorab})}\\
&=& \limsup_{r \to  \infty} \frac{1}{r} \log(\#\mathcal{F}^{\varepsilon} \cdot \#\mathcal{F}_r^{\varepsilon})\\
&=& \limsup_{r \to \infty} \frac{1}{r} \log (\#\mathcal{F}_r^{\varepsilon}) \qquad \qquad \text{($\#\mathcal{F}^{\varepsilon}$ in an constant)} \\
&\leq& \limsup_{r \to \infty} \frac{1}{r} \log\left(\frac{\vol \left(B(x,r+a+\frac{\varepsilon}{2})\right)}{C_{\varepsilon}}\right) \qquad \qquad \text{(by (\ref{Ball})\;)}\\
&=& \limsup_{r \to \infty} \frac{r+a+\frac{\varepsilon}{2}}{r} \cdot \frac{1}{r+a+\frac{\varepsilon}{2}} \log\left(\vol \left( B(x,r+a+\frac{\varepsilon}{2})\right)\right)\\
&=& \limsup_{\tilde r \to \infty} \frac{1}{\tilde r} \log\left(\vol \left(B(x,\tilde r)\right)\right)\qquad \text{(volume growth rate)}\\
&\leq & \limsup_{\tilde r \to \infty} \frac{1}{\tilde r} \log\left(\vol \left(B_E(x,A \tilde r)\right)\right)\qquad \text{($B_E$ denotes here the Euclidian ball)}\\
&\leq & \limsup_{\tilde r \to \infty} \frac{1}{\tilde r} \log (A^4 \tilde r^2) \\
&=& 0.  
\end{eqnarray*} 
\end{proof}

\begin{dfn}\label{expansive}
Let $(Y,d)$ be a compact metric space and $\phi^t:Y \to Y$ a continuous flow.
We will call the set 
$$S_{\mu}(v) = \{w \in Y \, |\, \sup_{t \in \Br}d(\phi^t(v), \phi^t(w)) \leq \mu\}$$ 
a {\it $\mu$-tube} for an element $v \in Y$.  \\
The topological entropy of $\phi^t$ restricted to $\mu$-tubes is given by
$$h_{\phi}^*(\mu):= \sup_{v \in Y}h_{top}(\phi, S_{\mu}(v)).$$
$\phi^t$ is called {\it $h$-expansive} if $h_{\phi}^*(\mu)=0$ for some $\mu>0$. For more details see~\cite{Bowen}.
\end{dfn}

\begin{theorem}\label{expansivitaet}
The geodesic flow $\tilde \phi^t$ on $\tilde ST^2$ is $h$-expansive for $\mu=\beta(D,A,a, \varepsilon)>0$
defined in Theorem~\ref{vieleStreifen}.
\end{theorem}

\begin{proof}
Let $v \in \tilde S T^2$ and consider the distance function $\bar d(v,w)=\max\limits_{s \in [0,1]} d(c_v(s), c_w(s))$ as defined in Remark~\ref{dist}. 
Then, 
$$\sup_{t \in \Br} \bar d(\tilde \phi^t(v), \tilde \phi^t(w)) = \sup_{t \in \Br} \max\limits_{s \in [0,1]} d(c_v(t+s), c_w(t+s)) = \sup_{t \in \Br} d(c_v(t), c_w(t)).$$
Hence the $\beta$-tube is given by
$$S_{\beta}(v) = \{w \in \tilde S T^2 \, |\, \sup_{t \in \Br} d(c_v(t), c_w(t)) \leq \beta\}.$$
We will show that $h_{top}(\tilde \phi, S_{\beta}(v))=0$. \\

\noindent
Minimal geodesics in the tube $S_{\beta}(v)$ have a bounded distance. Hence their accompanying straight lines are parallel.\\
This implies that all minimal geodesics in this tube are of the same type, i.e. they have the same rotation number. We distinguish the cases that the direction of the tube is rational or irrational. \\
1) First we consider the case that the direction is irrational. According to Bangert (see~\cite{B1988} or Remark~\ref{Hedlund}) all of the minimal geodesics have no intersections.\\
We fix $v \in \tilde S T^2$ and $\delta >0$ with $\delta \ll \min\{\beta, a\}$, where $a$ is the diameter of a fundamental area as introduced in the central construction. We denote by $F(T, \delta)$ a $(T, \delta)$-separated set for $S_{\beta}(v)$ with maximal cardinality and with respect to $\tilde \phi^t$ and $\bar d$. Let $w_1,w_2 \in F(T, \delta)$, i.e., $\bar d(w_1,w_2)_{T} \geq \delta$.
As the minimal geodesics $c_{w_1}$ and $c_{w_2}$ do not intersect, they are ordered if $c_{w_1}(\Br) \not= c_{w_2}(\Br)$. The special case that the  geodesics have a shifted parametrization but the same image, we will consider later.\\
As $w_1,w_2 \in F(T, \delta)$ there exists a time $t_0 \in [0,T+1]$ such that $d(c_{w_1}(t_0),c_{w_2}(t_0)) = l \geq \delta$.

We choose a minimal geodesic $\tilde c$ with length $l$ connecting $c_{w_1}(t_0)$ and $c_{w_2}(t_0)$. By construction it holds
$$\delta \leq d(c_{w_1}(t_0), c_{w_2}(t_0)) \leq 2\beta.$$
\\
We consider the geodesic triangle $\Delta(\delta)=\Delta(\delta, w_1, w_2):=\triangle (c_{w_1}(t_0), c_{w_2}(t_0- \frac{\delta}{2}), c_{w_2}(t_0+ \frac{\delta}{2}))$, where the sides of the triangle have minimal length.\\

\begin{figure}[h]
\begin{center}
\psfrag{cwt}{$c_{w_2}(t_0)$}
\psfrag{cvt}{$c_{w_1}(t_0)$}
\psfrag{cw}{$c_{w_2}$}
\psfrag{cv}{$c_{w_1}$}
\psfrag{c}{$q$}
\psfrag{a}{$\alpha$}
\psfrag{b}{$\beta$}
\psfrag{cwt-}{$c_{w_2}(t_0-\frac{\delta}{2})$}
\psfrag{cwt+}{$c_{w_2}(t_0+\frac{\delta}{2})$}
\includegraphics[scale=0.3]{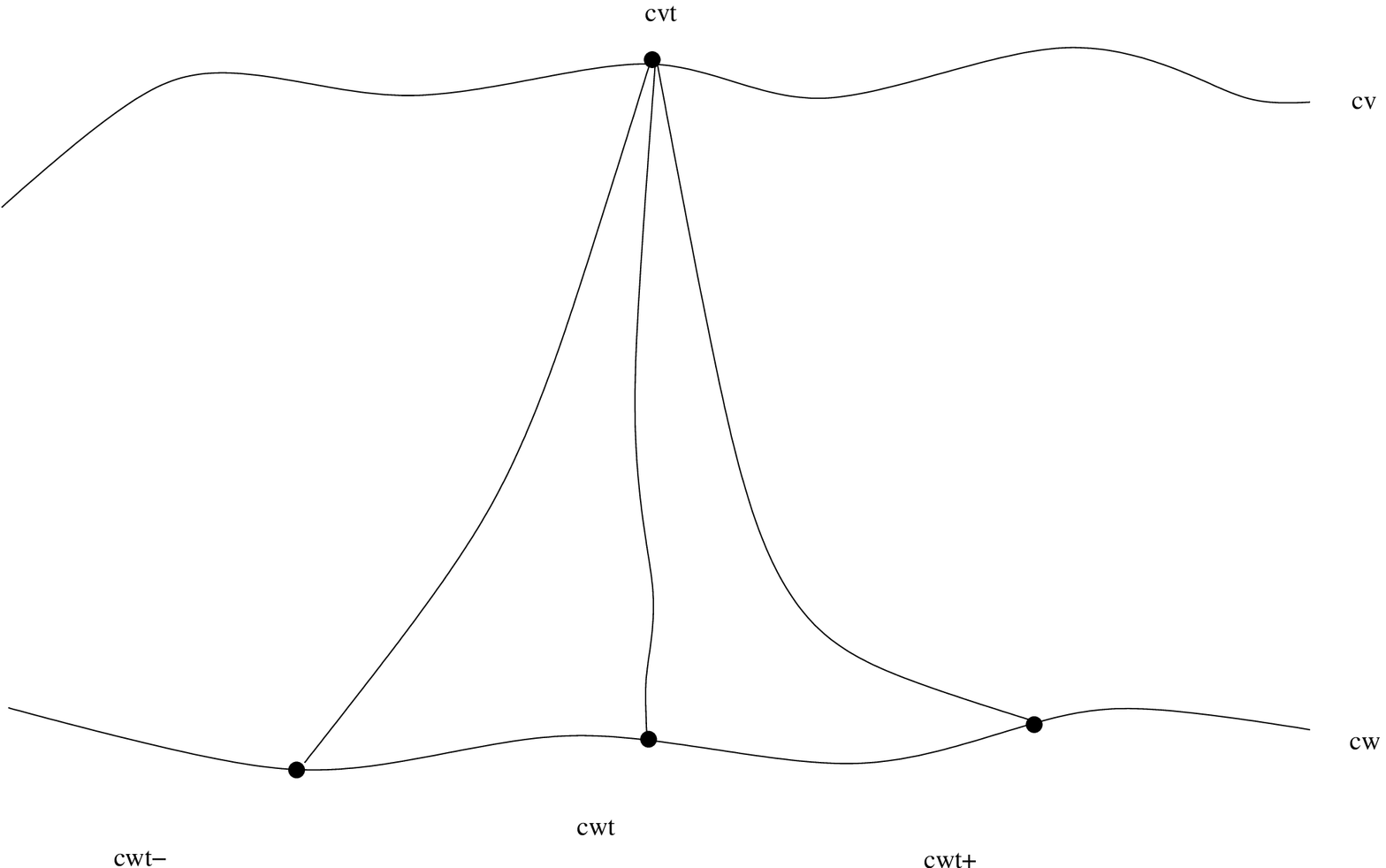}
\end{center}
\end{figure}
Obviously, both $c_{w_2}$ and $c_{w_1}$ intersect this triangle in no inner point. Otherwise, we would get a contradiction to the minimality of $c_{w_2}$ and $c_{w_1}$. Hence, the whole triangle lies in the strip between $c_{w_1}$ and $c_{w_2}$.\\
Let
$$N_{\beta+2 \delta,T}(v)=\{x \in \Br^2\;|\; d(x, c_v([0, T+1])) \leq \beta + 2 \delta\}$$
denote the $(\beta+2 \delta)$-neighborhood of $c_v([0, T+1])$. All geodesic segments $c_w([0, T+1])$ with $w \in F(T, \delta) \subset S_{\beta}(v)$ lie in $N_{\beta+2\delta, T}(v)$ and also the triangles $\Delta(\delta, w_1, w_2)$ for $w_1, w_2 \in F(T, \delta)$. Furthermore, 
there exists a constant $C_1>0$ such that the volume of $N_{\beta+2\delta, T}(v)$ is smaller than $C_1 \cdot \beta \cdot (T+1+2\beta +4 \delta)$. \\

\noindent
The main idea in the proof of the statement of this Theorem \ref{expansivitaet} is that the number of elements of $F(T, \delta)$ corresponds to the number of triangles $\Delta(\delta, w_1, w_2)$. And the number of the triangles is bounded by the quotient of the volume of $N_{\beta+2\delta, T}(v)$ and a lower bound of the volume of a single triangle.\\
For precise estimations we will need the following Lemma:
\begin{lemma} \label{c1c2}
Let $\vol(\Delta(\delta, w_1, w_2))$ denote the volume of the geodesic triangle $\Delta(\delta, w_1, w_2)$.
Then, there exists a constant $C_2(\delta)>0$ such that
$$\vol(\Delta(\delta, w_1, w_2)) \geq C_2(\delta).$$ 
\end{lemma}

\noindent
Now we will bound the number of $(T, \delta)$-separated minimal geodesics in the tube $S_{\beta}(v)$ with irrational direction.\\
Let $\tilde F(T, \delta)$ denote a $(T, \delta)$-separated set of maximal cardinality such that $c_{w_1}(\Br) \not=c_{w_2}(\Br)$ for $w_1\not= w_2 \in \tilde F(T, \delta)$.
Applying Lemma~\ref{c1c2} 
we get:
$$\#\tilde F(T, \delta) \leq \frac{\vol(N_{\beta+2\delta, T}(v))}{\min\limits_{w_1, w_2 \in \tilde F(T, \delta)}\vol(\Delta(\delta, w_1, w_2))} \leq \frac{C_1 \cdot\beta \cdot (T+1+2\beta + 4 \delta)}{C_2(\delta)}.$$
Now we consider minimal geodesics with the same image. Also they have to be separated and their number can be bounded by  $\frac{2\beta}{\delta}$. It follows for the whole $(T, \delta)$-separated set $F(T, \delta)$ 
with respect to the geodesic flow $\tilde \phi^t$ and the distance function $\bar d$:
$$\#F(T, \delta) \leq \# \tilde F(T, \delta) \cdot \frac{2\beta}{\delta}\leq \frac{C_1 \cdot \beta \cdot (T+1+2\beta + 4\delta) }{C_2(\delta)} \cdot \frac{2\beta}{\delta},$$
i.e., the growth of $\#F(T, \delta)$ is bounded by a linear function in $T$.\\
Let $I \subset \tilde ST^2$ denote the set of initial conditions $v$ such that the geodesics $c_v$ are minimal and have an irrational rotation number. Then,
$$ h^*_{\tilde \phi}(\beta) = \sup_{v \in \tilde ST^2|_I} h_{top}(\tilde \phi, S_{\beta}(v)) =\sup_{v \in \tilde ST^2|_I} \lim_{\delta \to 0} \limsup_{T \to \infty} \frac{1}{T} \log (\#F(T, \delta)) = 0.$$

2) As the second case we consider a tube $S_{\beta}(v)$ such that $v$ has a rational rotation number $\alpha$. Assume that a $(T, \delta)$-separated set $F(T, \delta)\subset S_{\beta}(v)$ of maximal cardinality fulfills the following properties:
\begin{enumerate}
\item[1)] For $w \in F(T, \delta)$ the corresponding minimal geodesic $c_w$ has rotation number $\alpha$.
\item[2)] $\# F(T, \delta)$ growths exponentially with $T$.
\end{enumerate}
Let $A_{\alpha}(T)=\{c_w\;|\; w \in F(T, \delta)\}$. The set of minimal geodesics with a fixed rational direction $\alpha$ consists of the three subsets $\mathcal{M}_{\alpha} ^{\operatorname{per}}, \mathcal{M}_{\alpha}^{\pm}$ and $\mathcal{M}_{\alpha}^{\mp}$ as introduced in Remark~\ref{Hedlund}. Then by assumption one of the ordered sets $A_{\alpha}(T) \cap \mathcal{M}_{\alpha} ^{\operatorname{per}}$, $A_{\alpha}(T) \cap \mathcal{M}_{\alpha}^{\pm}$ or $A_{\alpha}(T) \cap \mathcal{M}_{\alpha}^{\mp}$ growths exponentially with $T$. But by the same arguments as in the case of irrational directions we get a contradiction because the growth of the number of geodesics without intersections between each other is bounded by linear growth.
\end{proof}

\begin{proof}[\bf{Proof of Lemma~\ref{c1c2}}]
By the triangle-inequality we get that the length $l$ of the single sides of the triangle $\Delta(\delta, w_1, w_2)$ is bounded by
$$\frac{\delta}{2}\leq l \leq  2 \beta + \frac{\delta}{2}.$$ 
Assume there existed a triangle with a side of length $l$. Then we get an alternative way from $c_{w_1}(t_0)$ to $c_{w_2}(t_0)$ with length $\delta$ and a kink (if there were no kink, then $c_{w_1}$ and $c_{w_2}$ would intersect), in contradiction to the minimality of the geodesic segment $\tilde c$. Hence, especially $\frac{\delta}{2}< l \leq  2\beta + \frac{\delta}{2}$.
Assume that there does not exist a lower bound $C_2(\delta)>0$ for $\vol(\Delta(\delta, w_1, w_2))$.
Let a {\it min-$\delta$-triangle} be a triangle such that its sides are minimal geodesic segments, one side has length $\delta$, the length $l$ of the other sides is bounded by $\frac{\delta}{2}< l \leq 2\beta + \frac{\delta}{2}$ and their vertices lie on two minimal geodesics, i.e., the {\it min-$\delta$-triangle} are of the type of the triangles $\Delta(\delta, w_1, w_2)$.
We fix a compact set $C$ in $\Br^2$ large enough such that for each min-$\delta$-triangle $\triangle$ in $\Br^2$ there exists a translation element $\tau$ such that $\tau \triangle \in C$. 
We consider the set
$$D(\delta)=\{\triangle\;|\; \text{$\triangle \subset C$ is a min-$\delta$-triangle}\}.$$ 
Hence, all possible triangles belong to $D(\delta)$. 
As the constant $C_2(\delta)$ does not exist, there exists a sequence of triangles $\triangle_n \subset D(\delta)$ such that their corners converge and $\vol(\triangle_n)$ tends to zero. As $C$ is compact there exists a triangle $\tilde \triangle$ with $\vol(\tilde \triangle)=0$. Also for $\tilde \triangle$ it holds that $\frac{\delta}{2}< l$, otherwise we get a contradiction to the fact that the minimal geodesics through the vertices do not intersect.
But each non-degenerate geodesic triangle contains an open set which has positive volume, 
in contradiction to the existence of $\tilde \triangle$. Hence, $C_2(\delta)>0$ exists.

\end{proof}

\begin{theorem}[Bowen, see~\cite{Bowen}] \label{Theo-Bowen} 
Let $(Y,d)$ be a compact metric space and $f: Y \to Y$ a homeomorphism. Then
$$h_{top}(f) \leq h_{top}(f, \beta) + h^*_{f}(\beta).$$ In particular, $h_{top}(f) = h_{top}(f, \beta)$ if $\beta$ is an $h$-expansive constant for $f$.
\end{theorem}

\begin{proof}[\bf{Proof of the Main Theorem}]
By Theorem~\ref{vieleStreifen} and Theorem~\ref{expansivitaet} we get that $h_{top}(\tilde \phi, \beta)=0$ and $h^*_{\tilde \phi}(\beta)=0$. Applying the inequality of Theorem~\ref{Theo-Bowen} and extending it to continuous flows for $Y=\tilde ST^2$ and $f= \tilde \phi^t$ we conclude that $h_{top}(\tilde \phi) = 0$.
\end{proof}

\end{document}